\documentclass{amsproc}

\usepackage{amsmath,amsthm,amssymb,amsfonts,enumerate,color,url}

\newcommand{\abs}[1]{\left\vert#1\right\vert}
\newcommand{\set}[1]{\left\{#1\right\}}
\newcommand{\R}{\mathbb{R}}
\newcommand{\C}{\mathbb{C}}

\newcommand{\PV}{\operatorname{P.V.}}
\renewcommand{\L}{\mathcal{L}}

\newcommand{\T}{\mathbb{T}}

\newcommand{\Id}{\operatorname{I}}
\newcommand{\s}{\sigma}
\newcommand{\Dom}{\operatorname{Dom}}

\newcommand{\B}{\mathbb{B}}
\newcommand{\Nat}{\mathbb{N}}
\newcommand{\Z}{\mathbb{Z}}
\renewcommand{\S}{\mathbb{S}}
\newcommand{\be}{\begin{equation}}
\newcommand{\ee}{\end{equation}}

\newtheorem{theorem}{Theorem}[section]
\newtheorem{lemma}[theorem]{Lemma}

\theoremstyle{definition}

\theoremstyle{remark}
\newtheorem{remark}[theorem]{Remark}

\numberwithin{equation}{section}

\begin{document}

\title[Fractional Laplacians, Minakshisundaram and semigroups]{Fractional Laplacians on the sphere, \\
the Minakshisundaram zeta function and semigroups}


\author[P.~L.~De N\'apoli]{Pablo Luis De N\'apoli}
\address{IMAS (UBA-CONICET) and Departamento de Matem\'atica, Facultad
de Ciencias Exactas y Naturales, Universidad de Buenos Aires, Ciudad Universitaria, 1428 Buenos Aires, Argentina}
\curraddr{}
\email{pdenapo@dm.uba.ar}

\author[P.~R.~Stinga]{Pablo Ra\'ul Stinga}
\address{Department of Mathematics, Iowa State University, 396 Carver Hall, Ames, IA 50014, United States of America}
\curraddr{}
\email{stinga@iastate.edu}

\thanks{The first author was supported by ANPCyT under grant PICT 2014-1771 and by Universidad de Buenos Aires under grant 2002016010002BA. He is a member of Conicet, Argentina. The second author was partially supported by Grant MTM2015-66157-C2-1-P, MINECO/FEDER, EU,
from Government of Spain}



\date{}

\begin{abstract}
In this paper we show novel underlying connections between fractional powers
of the Laplacian on the unit sphere
and functions from analytic number theory and differential geometry,
like the Hurwitz zeta function and
the Minakshisundaram zeta function.
Inspired by Minakshisundaram's ideas, we find a precise pointwise description of
$(-\Delta_{\S^{n-1}})^s u(x)$ in terms of fractional powers of the Dirichlet-to-Neumann map
on the sphere. The Poisson kernel for the unit ball will be essential for this part of the analysis.
On the other hand, by using the heat semigroup on the sphere,
additional pointwise integro-differential formulas are obtained. Finally, we prove a characterization
with a local extension problem and the interior Harnack inequality.
\end{abstract}

\maketitle

\section{Introduction}

Nonlinear problems with fractional Laplacians have been receiving a lot of attention
for the past 12 years. Fractional nonlocal equations appear in several areas
of pure and applied mathematics, see for instance
\cite{Alonso-Sun, Caffarelli-Stinga, Caffarelli-Vasseur, Silvestre CPAM, Zaslavsky}.

For a function $u:\R^n\to\R$, $n\geq1$, the fractional Laplacian $(-\Delta)^s u$
with $0<s<1$ is naturally defined in a spectral way
by using the Fourier transform as
$$\widehat{(-\Delta)^s}u(\xi)=\abs{\xi}^{2s}\widehat{u}(\xi),\quad\xi\in\R^n.$$
The following equivalent semigroup formula holds:
$$(-\Delta)^s u(X)=\frac{1}{\Gamma(-s)}\int_0^\infty\big(e^{t\Delta}u(X)-u(X)\big)\,\frac{dt}{t^{1+s}},
\quad X\in\R^n.$$
Here $\Gamma$ is the Gamma function and $\{e^{t\Delta}\}_{t\geq0}$
is the classical heat semigroup on $\R^n$.
This implies the pointwise integro-differential formula
$$(-\Delta)^s u(X)=\frac{4^s\Gamma(n/2+s)}{\pi^{n/2}|\Gamma(-s)|}
\PV\int_{\R^n}\frac{u(X)-u(Y)}{|X-Y|^{n+2s}}\,dY,\quad X\in\R^n,$$
see \cite{Stinga-Torrea-CPDE}.
Clearly, $(-\Delta)^s$ is a nonlocal operator. The Caffarelli--Silvestre extension theorem \cite{Caffarelli-Silvestre CPDE}
establishes that if $U=U(X,y)$ is the solution to
$$\begin{cases}
\Delta U+\frac{1-2s}{y}\partial_yU+\partial_{yy}U=0,&\hbox{for}~X\in\R^n,~y>0,\\
U(X,0)=u(X),&\hbox{for}~X\in\R^{n},
\end{cases}$$
then
$$-y^{1-2s}\partial_yU(X,y)\big|_{y=0^+}=\frac{\Gamma(1-s)}{4^{s-1/2}\Gamma(s)}(-\Delta)^s u(X).$$
Moreover, the solution $U$ is given explicitly as (see \cite{Stinga-Torrea-CPDE})
\begin{align*}
U(X,y) &= \frac{y^{2s}}{4^s\Gamma(s)}\int_0^\infty e^{-y^2/(4t)}e^{t\Delta}u(X)\,\frac{dt}{t^{1+s}} \\
&= \frac{\Gamma(n/2+s)}{\pi^{n/2}\Gamma(s)}\int_{\R^n}\frac{y^{2s}}{(|X-Y|^2+y^2)^{\frac{n+2s}{2}}}u(Y)\,dY.
\end{align*}
See \cite{Caffarelli-Silvestre CPDE, Gale-Miana-Stinga, Stinga-Torrea-CPDE} for more details
about the extension problem and its applications.

In this paper we present several descriptions of the fractional powers of the Laplacian
$\Delta_{\S^{n-1}}$ on the unit sphere
$$\S^{n-1}=\set{X\in\R^n:\abs{X}=1},\quad n\geq2.$$
The Laplacian on the sphere is defined in the following simple way.
If $u=u(x)$ is a real function on $\S^{n-1}$ then we denote by $\tilde{u}$
the extension of $u$ to $\R^n\setminus\set{0}$ which coincides with $u$
on $\S^{n-1}$ and is constant along the lines normal to $\S^{n-1}$,
namely, $\tilde{u}(X)=u(X/|X|)$, $X\neq0$. Then $\Delta_{\S^{n-1}}u$ is the restriction
of the function $\Delta\tilde{u}$ to $\S^{n-1}$.
As a Riemannian manifold, the sphere has a natural
Laplace--Beltrami operator, that coincides with $\Delta_{\S^{n-1}}$ as obtained above.

Similarly to the fractional Laplacian on $\R^n$, the
fractional powers of $-\Delta_{\S^{n-1}}$ are defined in a spectral way.
Indeed, the spectral decomposition of the Laplacian on the sphere
is given by the spherical harmonics, that we briefly describe
next (see \cite{Folland, Seeley}). We 
have $L^2(\S^{n-1})=\bigoplus_{k=0}^\infty SH^k$, where $SH^k$
denotes the space of spherical harmonics of degree $k\geq0$
of dimension $d_k=\frac{(2k+n-2)(k+n-3)!}{k!(n-2)!}$.
From now on and for the rest of the paper we fix an orthonormal basis
$\set{Y_{k,l}:1\leq l\leq d_k}$ of $SH^k$ consisting of real spherical harmonics
$Y_{k,l}$ of degree $k\geq0$. The spherical harmonics are the eigenfunctions of the Laplacian on the sphere, namely,
$$-\Delta_{\S^{n-1}}Y_{k,l}(x)=\lambda_kY_{k,l}(x),\quad x\in\S^{n-1},$$
with eigenvalues
$$\lambda_k=k(k+n-2),$$
of multiplicity $d_k$, for $k\geq0$.
If $u$ has an expansion into spherical harmonics as
\be\label{decomp f}
u(x)=\sum_{k=0}^\infty\sum_{l=1}^{d_k}c_{k,l}(u)Y_{k,l}(x),\quad
c_{k,l}(u)=\int_{\S^{n-1}}u(y)Y_{k,l}(y)\,d\mathcal{H}^{n-1}(y),
\ee
where $d\mathcal{H}^{n-1}$ is the $(n-1)$-dimensional Hausdorff measure restricted to $\S^{n-1}$,
then for any $\s\in\mathbb{C}$ such that $s=\mathrm{Re}(\s)>0$, we can define
\be\label{spectral-formula}
(-\Delta_{\S^{n-1}})^{\pm\s}u(x)= \sum_{k=0}^\infty\lambda_k^{\pm\s}\sum_{l=1}^{d_k}c_{k,l}(u)Y_{k,l}(x).
\ee
The orthogonal projector onto the subspace $SH^k$ of $L^2(\S^{n-1})$ is given by
\begin{equation}\label{projector}
P_ku(x)=\sum_{l=1}^{d_k}c_{k,l}(u)Y_{k,l}(x).
\end{equation}
Hence
\begin{equation}\label{spectral-formula2}
(-\Delta_{\S^{n-1}})^{\pm\s}u(x)=\sum_{k=0}^\infty\lambda_k^{\pm\s}P_ku(x),
\end{equation}
so the definition of $(-\Delta_{\S^{n-1}})^{\pm\s}u$ is independent of the basis.

Some care is needed for the formulas above to be well defined. For $(-\Delta_{\S^{n-1}})^{-\s}u$
to make sense we need $\displaystyle\int_{\S^{n-1}}u=0$. This condition
says that the term $P_0u(x)$, which corresponds to the eigenvalue $\lambda_0=0$ in
\eqref{spectral-formula} and \eqref{spectral-formula2}, vanishes,
so we may consider the sums as starting from $k=1$.
On the other hand, $(-\Delta_{\S^{n-1}})^\s u\in L^2(\S^{n-1})$ if and only if
$u\in\Dom((-\Delta_{\S^{n-1}})^\s)\equiv\big\{u\in L^2(\S^{n-1}):\sum_{k=0}^\infty\lambda_k^{2s}
\sum_{l=1}^{d_k}\abs{c_{k,l}(u)}^2<\infty\big\}$.
In particular, \eqref{spectral-formula} and \eqref{spectral-formula2}
make sense for functions $u\in C^\infty(\S^{n-1})$,
see \eqref{eq:decaycoefficients}.
Here $u\in C^m(\S^{n-1})$, $m\in\Nat_0\cup\set{\infty}$, means that the extension $\tilde{u}$
of $u$ along normal lines to the sphere is a $C^m$ function in a neighborhood of the sphere.
In general, $(-\Delta_{\S^{n-1}})^\s u$ is defined as a distribution in $H^{-\s}(\S^{n-1})$,
for every $u\in H^\s(\S^{n-1})$. In addition, if $u\in H^{-\s}(\S^{n-1})$ satisfies $\langle u,1\rangle=0$ then
$(-\Delta_{\S^{n-1}})^{-\s}u\in H^{\s}(\S^{n-1})$ with $\displaystyle\int_{\S^{n-1}}(-\Delta_{\S^{n-1}})^{-\s}u=0$.
In this paper we are mainly interested in describing pointwise formulas
and the extension problem for these fractional operators in connection
with functions from number theory and differential geometry, and semigroups.
Therefore we will always assume that all our functions are smooth.

The fractional Laplacian on the sphere is a natural object to consider
since it is the simplest example of a fractional nonlocal operator on a compact Riemannian manifold.
Moreover, the sphere has a rich structure related to the fact that
it is a homogenous space under the action of the Lie group $SO(n)$.
Such a rich theory and some clever formulas connected with the Minakshisundaram
zeta function in combination with the method of semigroups
will allow us to obtain rather explicit expressions for
the fractional operators in terms of the Dirichlet-to-Neumann map
and the heat semigroup on the sphere.
For applications, see for example \cite{Alonso-Sun}.
For hypersingular integrals and potential operators on the sphere
(which are not the same as the fractional powers of the Laplacian
on the sphere) the reader can consult \cite{Rubin,Samko}.
Certainly one can define fractional powers of Laplace--Beltrami operators in Riemannian
manifolds by using the spectral theorem. In those cases in which estimates for the corresponding heat semigroup
kernel are available, one could try to derive pointwise expressions and kernel estimates
for fractional operators by using the method of semigroups
from \cite{Gale-Miana-Stinga, Stinga-Torrea-CPDE},
in an analogous way as explained here for the case of the sphere (see section \ref{Section:extension}).
Observe that for this general procedure to apply we do not need the manifold to be necessarily compact.

A description of the contents of the paper follows.
\begin{enumerate}[(a)]
\item In sections \ref{Section:One-dimensional} and \ref{Section:One-dimensional-positive}
we first consider the cases of the
negative and positive powers of the Laplacian on the unit circle, respectively, in connection with the
Hurwitz zeta function the functions defined by Fine in \cite{Fine}.
\item In section \ref{Section:Tools} we define the
Dirichlet-to-Neumann operator $L$ for the Laplacian in the unit ball.
The semigroup generated by $L$ is obtained from the Poisson kernel for the ball.
We show how this kernel can be deduced by using the Funk--Hecke identity
and the generating formula for Gegenbauer polynomials.
\item In section \ref{Section:Negative} we relate the 
negative powers of $-\Delta_{\S^{n-1}}$ in dimension $n\geq3$
with the Minakshisundaram Riemann zeta function on the sphere.
We use this connection and the Poisson kernel to find precise estimates
for the kernel of $(-\Delta_{\S^{n-1}})^{-s}$.
\item In section \ref{Section:Positive}, inspired by Minakshisundaram's ideas,
we prove a numerical formula (Lemma \ref{Lem:Minak positive}) that 
permits us to express the fractional Laplacian
on the sphere $(-\Delta_{\S^{n-1}})^s$ in terms of the fractional
Dirichlet-to-Neumann operator $(L+\frac{n-2}{2})^{2s}$.
Precise kernel estimates are shown.
\item In section \ref{Section:extension} we use the heat semigroup on the sphere
to give equivalent formulas for $(-\Delta_{\S^{n-1}})^{\pm s}u(x)$ with kernel estimates.
We also present the extension problem characterization and,
as an application, the interior Harnack inequality.
\end{enumerate}

\noindent\textbf{Notation.} Throughout the paper $\s\in\C_+$
denotes a complex number with positive real part
$\mathrm{Re}(\s)=s>0$. For two positive quantities $A$ and $B$ we write $A\sim B$
to mean that there exist constants $c_{n,\s}$ and $C_{n,\s}$ such that $c_{n,\s}\leq A/B\leq C_{n,\s}$.
If $x,y\in\S^{n-1}$ then $d(x,y)=\cos^{-1}(x\cdot y)$
is the geodesic distance between $x$ and $y$. Notice that $|x-y|^2=2(1-x\cdot y)\sim d(x,y)^2$,
for all $x,y\in\S^{n-1}$.
We will always denote by $u$ a smooth real function on the sphere.

\section{Negative powers in the unit circle and the Hurwitz zeta function}\label{Section:One-dimensional}

In this section and the next one
we consider the one-dimensional case of the unit circle $\T=\S^1\subset\R^2$. As usual, we identify functions on
$\T$ with periodic functions on the interval $[0,1]$, so that $\Delta_{\T}=\frac{d^2}{dx^2}$.
The spherical harmonics become the complex exponentials $\{e^{2\pi ikx}\}_{k\in\Z}$, $x\in[0,1]$.
We have $-\Delta_{\T}e^{2\pi ikx}=(2\pi k)^2e^{2\pi ikx}$.

When $\displaystyle\int_{\T}u=0$ and
$\s\in\C_+$ with $s=\mathrm{Re}(\s)>0$, we can write
\begin{align*}
(-\Delta_\T)^{-\s/2}u(x) &= \sum_{k\in\Z\setminus\{0\}}(2\pi|k|)^{-\s}c_k(u)e^{2\pi ikx} \\
&= \sum_{k\in\Z\setminus\{0\}}(2\pi|k|)^{-\s}\int_0^1u(y)e^{2\pi ik(x-y)}\,dy \\
&=\int_0^1K_{-\s}(x-y)u(y)\,dy,
\end{align*}
where the kernel is given by
\begin{equation}\label{powers-kernel-S1}
K_{-\s}(x)=\sum_{k\in\Z\setminus\set{0}}\frac{e^{2\pi ikx}}{(2\pi|k|)^\s}.
\end{equation}

Our main interest will be to understand the behavior of $K_{-\s}(x)$ as $x\to0,1$.
Estimates for this kernel can be obtained, for example, by using the semigroup
method of \cite{Stinga-Torrea-CPDE} or the transference principle from \cite{Roncal-Stinga2}.
Nevertheless, here we take a different approach by connecting formula \eqref{powers-kernel-S1}
with some useful functions from analytic number theory.

The Hurwitz zeta function (see \cite[Chapter~12]{Apostol}, also \cite{Fine, online})
is initially defined for $0<x\leq1$ and $\s\in\C_+$
such that $s=\mathrm{Re}(\s)>1$ by
\begin{equation}\label{eq:Hurwitzzeta}
\zeta(\s,x)=\sum_{k=0}^\infty\frac{1}{(k+x)^\s}.
\end{equation}
When $x=1$ this reduces to the Riemann zeta function $\zeta(\s)=\zeta(\s,1)$.
Moreover,
\begin{equation}\label{eq:trivial}
\zeta(\s,x)=x^{-\s}+\zeta(\s,x+1).
\end{equation}
Using the Gamma function it is possible to extend $\zeta(\s,x)$ as an analytic function of $\s\in\C$
except for a simple pole at $\s=1$, see \cite[Theorem~12.4]{Apostol}.

A related function is the periodic zeta function defined
for $x\in\R$ and $s>1$ by
$$F(x,\s)=\sum_{k=1}^\infty\frac{e^{2\pi ikx}}{k^\s}.$$
This series converges absolutely if $s>1$. If $x$ is not an integer the series 
also converges conditionally whenever $s>0$.
Notice that $F(x,\s)$ is a periodic function with period 1 that
coincides with the Riemann zeta function $\zeta(\s)$
at $x=1$. It is clear then that
$$K_{-\s}(x)=\frac{1}{(2\pi)^\s}\big[F(x,\s)+F(-x,\s)\big].$$

The two functions $\zeta(\s,x)$ and $F(x,\s)$ are related by the following formula due to Hurwitz,
see \cite[Theorem~12.6]{Apostol}, also \cite{Fine}: if $0<x\leq1$ and $s>1$ then
$$\zeta(1-\s,x)=\frac{\Gamma(\s)}{(2\pi)^\s}\big[e^{-\pi i\s/2}F(x,\s)+e^{\pi i\s/2}F(-x,\s)\big].$$
If $x\neq1$ this representation is also valid for $s>0$.
If we substitute $x$ by $1-x$, then
$$\zeta(1-\s,1-x)=\frac{\Gamma(\s)}{(2\pi)^\s}\big[e^{-\pi i\s/2}F(-x,\s)+e^{\pi i\s/2}F(x,\s)\big].$$
Therefore
$$\zeta(1-\s,x)+\zeta(1-\s,1-x)
=\frac{\Gamma(\s)}{(2\pi)^\s}2\cos\big(\tfrac{\pi\s}{2}\big)\big[F(x,\s)+F(-x,\s)\big],$$
or
\begin{equation}\label{kernel-S1}
K_{-\s}(x)=\frac{1}{2\Gamma(\s)\cos\left(\frac{\pi\s}{2}\right)}\big[\zeta(1-\s,x)+\zeta(1-\s,1-x)\big].
\end{equation}

To estimate $K_{-\s}(x)$ we use \eqref{kernel-S1} and well known asymptotic
expansions for the Hurwitz zeta function. Recall that $\zeta(1-\s,1)=\zeta(1-\s)$. For $\s\neq1$ fixed,
$$\zeta(\s,x+1)=\zeta(\s)-\s\zeta(\s+1)x+O(x^2),\quad\hbox{as}~x\to0,$$
see \cite{online}. By replacing $\s$ by $1-\s$ in \eqref{eq:trivial} and using the asymptotic expansion above,
\begin{align*}
\zeta(1-\s,x) &= x^{\s-1}+\zeta(1-\s,x+1) \\
&= x^{\s-1}+\zeta(1-\s)-(1-\s)\zeta(2-\s)x+O(1),
\end{align*}
as $x\to0$. If we substitute $x$ by $1-x$ in the latter expansion, we get
$$\zeta(1-\s,1-x)=(1-x)^{\s-1}+\zeta(1-\s)-(1-\s)\zeta(2-\s)(1-x)+O(1),$$
as $x\to1$.
By plugging these estimates into \eqref{kernel-S1} we deduce the asymptotic formulas
$$K_{-\s}(x)=C_\s x^{\s-1}+O(1),\quad\hbox{as}~x\to0,$$
and
$$K_{-\s}(x)=C_\s(1-x)^{\s-1}+O(1),\quad\hbox{as}~x\to1.$$
We conclude that
$$K_{-\s}(x)\sim\frac{1}{x^{1-\s}}+\frac{1}{(1-x)^{1-\s}},\quad\hbox{as}~x\to0,1.$$

\section{Positive powers in the unit circle and the Hurwitz zeta \\
and Fine functions}\label{Section:One-dimensional-positive}

In this section we continue our analysis of the one-dimensional case of the unit circle $\T=\S^1\subset\R^2$,
$\T\equiv[0,1]$. We study the kernel of the fractional power operator
$(-\Delta_\T)^{\s/2}$, when $\s\in\C_+$ is such that $0<\mathrm{Re}(\s)<2$.
As in the previous section we denote by $\zeta$ the Hurwitz zeta function \eqref{eq:Hurwitzzeta}.

The key idea is to use the heat semigroup on the circle. This approach
for the case of the sphere will be developed in detail in section \ref{Section:extension}.
It is easy to check that the proof of Theorem \ref{thm:fractionalandheat}
is valid when $n=2$ and $s$ is replaced by $\s/2$,
see also \cite{Gale-Miana-Stinga, Roncal-Stinga1, Roncal-Stinga2}. Then, for any $x\in\T$,
$$(-\Delta_\T)^{\s/2}u(x)=\PV\int_\T\big(u(y)-u(x)\big)K_{\s}(x-y)\,dy,$$
where the kernel $K_\s$ is given by
$$K_\s(x)=\frac{1}{\Gamma(-\s/2)}\int_0^\infty W_t(x)\,\frac{dt}{t^{1+\s/2}}.$$
Here $W_t(x)$ is the heat kernel in the unit circle:
\begin{equation}\label{Fine:1}
\begin{aligned}
W_t(x) &= \sum_{k\in\Z}e^{-4\pi^2k^2t}e^{2\pi ikx} \\
&= 1+2\sum_{k=1}^\infty e^{-4\pi^2k^2t}\cos(2\pi kx).
\end{aligned}
\end{equation}
Notice that this kernel is essentially the so-called Jacobi theta function.

By following Fine \cite{Fine}, we define, for $0<x<1$ and $t>0$, the function
$$f(x,t)=1+2\sum_{k=1}^\infty e^{-\pi k^2t}\cos(2\pi k x).$$
From \eqref{Fine:1} it follows that
\begin{equation}\label{Fine:2}
W_t(x)=f(x,4\pi t).
\end{equation}
Next, Fine observes in \cite{Fine} that the functions
$$F(x,\omega)=\int_0^1f(x,t)t^{\omega/2-1}\,dt$$
and
$$G(x,\omega)=\int_1^\infty\big(f(x,t)-1\big)t^{\omega/2-1}\,dt$$
are entire functions in $\omega$. Additionally, Fine introduces the function
$$H(x,\omega)=F(x,\omega)+G(x,\omega)-\frac{2}{\omega}$$
and proves the following relation with the Hurwitz zeta function:
if $\mathrm{Re}(\omega)<0$ then
\begin{equation}\label{Fine:3}
H(x,\omega)=\frac{\Gamma(\frac{1-\omega}{2})}{\pi^{\frac{1-\omega}{2}}}\big[\zeta(1-\omega,x)+\zeta(1-\omega,1-x)\big],
\end{equation}
see \cite[(11)]{Fine}.
Thus, by using \eqref{Fine:2}, the change of variables $4\pi t\to t$ and the functions $F$, $G$ and
$H$, together with the functional identity \eqref{Fine:3}, we find that
\begin{align*}
K_\s(x) &=\frac{(4\pi)^{\s/2}}{\Gamma(-\s/2)}\int_0^\infty f(x,t)\,\frac{dt}{t^{1+\s/2}} \\
&=\frac{(4\pi)^{\s/2}}{\Gamma(-\s/2)}\bigg[\int_0^1f(x,t)\,\frac{dt}{t^{1+\s/2}}+
\int_1^\infty\big(f(x,t)-1\big)\,\frac{dt}{t^{1+\s/2}}+\frac{2}{\s}\bigg] \\
&=\frac{(4\pi)^{\s/2}}{\Gamma(-\s/2)}\bigg[F(x,-\s)+G(x,-\s)+\frac{2}{\s}\bigg] \\
&=\frac{(4\pi)^{\s/2}}{\Gamma(-\s/2)}H(x,-\s) \\
&=\frac{4^{\s/2}\Gamma(\frac{1+\s}{2})}{\Gamma(-\s/2)\sqrt{\pi}}\big[\zeta(1+\s,x)+\zeta(1+\s,1-x)\big].
\end{align*}
We can use the following properties of the Gamma function
\begin{equation}\label{Gamma}
\Gamma(\tfrac{1}{2}+z)\Gamma(\tfrac{1}{2}-z)=\frac{\pi}{\cos(\pi z)}\qquad
\Gamma(2z)=\frac{2^{2z-1}}{\sqrt{\pi}}\Gamma(z)\Gamma(z+\tfrac{1}{2})
\end{equation}
with $z=-\s/2$ to further simplify the expression above to
$$K_\s(x)=\frac{1}{2\Gamma(-\s)\cos(-\frac{\pi\s}{2})}\big[\zeta(1+\s,x)+\zeta(1+\s,1-x)\big].$$
This is perfectly consistent with \eqref{kernel-S1}.
By performing an asymptotic analysis with of the Hurwitz zeta function
analogous to the one we did in section \ref{Section:One-dimensional},
we are able to conclude that
$$K_{\s}(x)\sim\frac{1}{x^{1+\s}}+\frac{1}{(1-x)^{1+\s}},\quad\hbox{as}~x\to0,1.$$

\section{The semigroup generated by the Dirichlet-to-Neumann map}\label{Section:Tools}

In this section we introduce the Dirichlet-to-Neumann map $L$ for the Laplacian in the unit ball
$\B=\{X\in\R^n:|X|<1\}$, $n\geq2$. We show that the kernel of
the semigroup generated by $L$ is obtained from the Poisson kernel
for the unit ball. These objects will be very useful in the description of the fractional operators
in sections \ref{Section:Negative} and \ref{Section:Positive}.

Let $u$ be a smooth function on the sphere with series expansion \eqref{decomp f}.
We define the Dirichlet-to-Neumann operator $L$ on the sphere $\S^{n-1}$ by
$$Lu(x)=\sum_{k=0}^\infty k\sum_{l=1}^{d_k}c_{k,l}(u)Y_{k,l}(x)
=\sum_{k=0}^\infty kP_ku(x),$$
for $x\in\S^{n-1}$,
where $P_k$ is the orthogonal projector \eqref{projector}.
The series is absolutely convergent and can be differentiated term by term. Indeed,
for each multi-index $\alpha=(\alpha_1,\ldots,\alpha_n)\in\mathbb{N}_0^n$, there exists $C_{\alpha,n}>0$ such that
\begin{equation}\label{acotacion esf arm}
\abs{D^\alpha Y_{k,l}(X/\abs{X})}^2\leq (C_{\alpha,n})^2k^{2\abs{\alpha}+n-2},\quad\hbox{for all}~1\leq l\leq d_k,
\end{equation}
where $D^\alpha=\tfrac{\partial^{\abs{\alpha}}}{\partial X_1^{\alpha_1}\cdots\partial X_n^{\alpha_n}}$,
$\abs{\alpha}=\alpha_1+\cdots+\alpha_n$, see \cite{Seeley}. Then, for any $m\in\Nat$, by
the symmetry of $\Delta_{\S^{n-1}}$ and the Cauchy-Schwartz inequality,
\begin{equation}\label{eq:decaycoefficients}
\begin{aligned}
\abs{c_{k,l}(u)} &= \frac{1}{\lambda_k^m}\bigg|\int_{\S^{n-1}}(-\Delta_{\S^{n-1}})^mu(x)Y_{k,l}(x)\,d\mathcal{H}^{n-1}(x)\bigg| \\
&\leq \frac{C_{u,m}}{(k(k+n-2))^m}.
\end{aligned}
\end{equation}
Therefore the estimates in \eqref{acotacion esf arm} and \eqref{eq:decaycoefficients} give the conclusion.

The semigroup $\set{e^{-tL}}_{t\geq0}$ generated by the Dirichlet-to-Neumann operator
is related to the solution $w=w(X)$ to the Dirichlet problem in the unit ball
$$\begin{cases}
    \Delta w=0, & \hbox{in}~\B, \\
    w=u, & \hbox{on}~\S^{n-1},
  \end{cases}
$$
in the following way. The solution $w$ can be recovered from
$u$ by using the Poisson integral formula for the unit ball:
\begin{equation}\label{Poisson bola}
w(X)=\int_{\S^{n-1}}\frac{1-\abs{X}^2}{\omega_{n-1}(1-2X\cdot y+\abs{X}^2)^{n/2}}u(y)\,d\mathcal{H}^{n-1}(y),
\end{equation}
where $\omega_{n-1}=\frac{2\pi^{n/2}}{\Gamma(n/2)}$ is the surface area of $\S^{n-1}$.
On the other hand, if we introduce polar coordinates $X=rx$, $0<r\leq1$, $x\in\S^{n-1}$, then,
$$w(X)=\sum_{k=0}^\infty r^k\sum_{l=1}^{d_k}c_{k,l}(u)Y_{k,l}(x)=\sum_{k=0}^\infty r^kP_ku(x).$$
See \cite{Folland} for details. Now, 
by making the change of parameters $r=e^{-t}$, $t\geq0$, we can regard $w(X)$ as a function of
$t\geq0$ and $x\in\S^{n-1}$:
\be\label{Poisson sum}
e^{-tL}u(x)\equiv w(X)=\sum_{k=0}^\infty e^{-tk}\sum_{l=1}^{d_k}c_{k,l}(u)Y_{k,l}(x)=
\sum_{k=0}^\infty e^{-tk}P_ku(x).
\ee
This is in fact the heat-diffusion semigroup generated by $L$. Indeed,
we can differentiate the series in \eqref{Poisson sum} to get
$$-\partial_t e^{-tL}u(x)\big|_{t=0}=r\partial_rw(X)\big|_{r=1}=\partial_\nu w(X)=Lu(x),$$
which also shows that the name ``Dirichlet-to-Neumann operator'' for $L$ is fully justified.
Also, $e^{-tL}u\to u$, as $t\to0$, uniformly and in
$L^p(\S^{n-1})$, for $1\leq p<\infty$.

By letting $r=e^{-t}$ in \eqref{Poisson bola},
we find that the semigroup $e^{-tL}$ admits an expression
as a convolution on the sphere with the Poisson kernel:
\begin{equation}\label{Poisson semigroup}
\begin{aligned}
e^{-tL}u(x) &=\int_{\S^{n-1}}\frac{1-e^{-2t}}{\omega_{n-1}
(1-2e^{-t}x\cdot y+e^{-2t})^{n/2}}u(y)\,d\mathcal{H}^{n-1}(y) \\
&= \int_{\S^{n-1}}\frac{1-e^{-2t}}{\omega_{n-1}((1-e^{-t})^2+2e^{-t}(1-x\cdot y))^{n/2}}
u(y)\,d\mathcal{H}^{n-1}(y) \\
&\equiv\int_{\S^{n-1}}P_{e^{-t}}(x\cdot y)u(y)\,d\mathcal{H}^{n-1}(y).
\end{aligned}
\end{equation}
 
\begin{lemma}\label{Prop:bound t}
There exists a constant $C>0$ depending only on
$u$ and $n$ such that $|e^{-tL}u(x)-u(x)|\leq Ct$, for any $0<t<1$, for all $x\in\S^{n-1}$.
\end{lemma}

\begin{proof}
By the mean value theorem,
$$|e^{-tL}u(x)-u(x)|=|e^{-tL}u(x)-e^{-0L}u(x)|=|\partial_te^{-tL}u(x)|\big|_{t=\xi}t,$$
where $\xi$ is an intermediate point between $0$ and $t$. But then, using \eqref{Poisson sum}
together with \eqref{acotacion esf arm} and \eqref{eq:decaycoefficients}, it readily follows
that $\sup_{t,x}|\partial_te^{-tL}u(x)|\leq C$.
\end{proof}

To relate $L$ with the Laplacian on the sphere we
recall that the eigenvalues of  $-\Delta_{\S^{n-1}}$ are $\lambda_k=k(k+n-2)$.
Observe that, unlike the case of the classical Poisson integral in the upper half-plane,
the Poisson integral in the ball \eqref{Poisson bola} is \textit{not} the Poisson semigroup
$\{e^{-t(-\Delta_{\S^{n-1}})^{1/2}}\}_{t\geq0}$ generated by $-\Delta_{\S^{n-1}}$. In fact,
$e^{-t(-\Delta_{\S^{n-1}})^{1/2}}\neq e^{-tL}$, except for the unit circle, that is, if $n=2$. Moreover,
$$-\Delta_{\S^{n-1}}=L(L+(n-2)\Id),$$
where $\Id$ is the identity operator. It then follows that
$$(-\Delta_{\S^{n-1}})^{\pm\s}=L^{\pm\s}(L+(n-2)\Id)^{\pm\s},\quad\s\in\C_+.$$

The explicit formula for the Poisson kernel in \eqref{Poisson semigroup}
comes from the generating formula for the Gegenbauer (or ultraspherical) polynomials.
The Funk--Hecke theorem (see \cite{Seeley}) states that if $F(\tau)(1-\tau^2)^{(n-3)/2}$
is an integrable function on the interval $(-1,1)$ then,
for each spherical harmonic $Y_{k,l}$ and $x\in\S^{n-1}$,
\begin{equation}\label{Funk--Hecke}
\begin{aligned}
\int_{\S^{n-1}}F(x\cdot y)Y_{k,l}(y)&\,d\mathcal{H}^{n-1}(y)\\
&=Y_{k,l}(x)\frac{\omega_{n-2}}{C_k(1)}\int_{-1}^1F(\tau)C_k(\tau)(1-\tau^2)^{(n-3)/2}\,d\tau,
\end{aligned}
\end{equation}
where $C_k(\tau)$ is the Gegenbauer
polynomial $C_k^\nu(\tau)$ with parameter $\nu=(n-2)/2$. For $n=2$ they are the
Chebyshev polynomials and for $n=3$ they are the Legendre polynomials
(see \cite{Folland,Lebedev,online}).
Using the generating formula
$$(1-2\tau r+r^2)^{-\nu}=\sum_{k=0}^\infty C_k^\nu(\tau)r^k,$$
see \cite[eq.~(5.12.7)]{Lebedev}, it easily follows that
\begin{equation}\label{generating-formula}
\sum_{k=0}^\infty\frac{k+\nu}{\nu}\,C^\nu_k(\tau)r^k=\frac{1-r^2}{(1-2r\tau+r^2)^{\nu+1}}.
\end{equation}
For each fixed $x\in\S^{n-1}$, $C_k(x\cdot y)$ (as a function of $y\in\S^{n-1}$)
is in $SH^k$, see \cite{Folland}. By expressing $C_k(x\cdot y)$ in terms of the orthonormal basis
$\{Y_{k,l}:l=1,\ldots,d_k\}$ of $SH^k$, and applying 
the Funk--Hecke formula \eqref{Funk--Hecke} in combination with
properties of Gegenbauer polynomials and the Gamma function, we can see that
\begin{equation}\label{eq:super identity}
\frac{1}{\omega_{n-1}}\cdot\frac{k+\frac{n-2}{2}}{\frac{n-2}{2}}\,C_k(x\cdot y)=\sum_{l=1}^{d_k}Y_{k,l}(x)Y_{k,l}(y).
\end{equation}
Plugging this into the first sum in \eqref{Poisson sum} we get
$$e^{-tL}u(x)=\int_{\S^{n-1}}\frac{1}{\omega_{n-1}}\bigg(\sum_{k=0}^\infty
\frac{k+\frac{n-2}{2}}{\frac{n-2}{2}}\,C_k(x\cdot y)r^k\bigg)u(y)\,d\mathcal{H}^{n-1}(y).$$
Thus \eqref{Poisson semigroup} with $e^{-t}=r$ follows by using the generating formula \eqref{generating-formula}
with $\nu=\frac{n-2}{2}$.

\section{Negative powers and the Minakshisundaram
zeta function}\label{Section:Negative}

Throughout this section we always assume that $\displaystyle\int_{\S^{n-1}}u=0$.
Recall that we have set $s=\mathrm{Re}(\s)>0$ and that the eigenvalues of $-\Delta_{\S^{n-1}}$
are $\lambda_k=k(k+n-2)$, $k\geq0$.
By using \eqref{eq:super identity} into \eqref{spectral-formula}
we find that the negative powers of the Laplacian on the sphere 
have an integral representation as an spherical convolution
\begin{equation}\label{integral-formula}
(-\Delta_{\S^{n-1}})^{-\s}u(x)=\int_{\S^{n-1}}K_{-\s}(x\cdot y)u(y)\,d\mathcal{H}^{n-1}(y),
\end{equation}
where the kernel is given in terms of the Gegenbauer polynomials as
\begin{equation}\label{MS-zeta}
K_{-\s}(x\cdot y)=\frac{1}{\omega_{n-1}}
\sum_{k=1}^\infty\lambda_k^{-\s}\cdot\frac{k+\frac{n-2}{2}}{\frac{n-2}{2}}\,C_k(x\cdot y).
\end{equation}

The series in \eqref{MS-zeta} is the zeta function on the sphere of S. Minakshisundaram,
see \cite{Minakshisundaram}\footnote{We have a slightly different normalization of Gegenbauer polynomials
with respect to Minakshisundaram. Compare \eqref{generating-formula} with
\cite[Lemma~1]{Minakshisundaram}. Nevertheless, the zeta functions coincide up to the
normalizing factor $\omega_{n-1}$.}.
In \cite{Minakshisundaram} Minakshisundaram analyzed this function with methods similar to those used
in analytic number theory for studying the Riemann zeta function.
Using that $|C_k(x\cdot y)|\leq |C_k(1)|={{k+n-3}\choose{k}}$ we get
$$C_k(x\cdot y)=O(k^{n-3}).$$
From here it is easy to see that the Dirichlet series \eqref{MS-zeta}
converges absolutely and uniformly in compact sets
when $s=\hbox{Re}(\s)>\frac{n-1}{2}$. For this range of $s$, $K_{-\s}(x\cdot y)$
represents an analytic function of $\s$.
However, Minakshisundaram showed that (as it happens with Riemann zeta function) the function
$K_{-\s}(x\cdot y)$ can be continued as a meromorphic function of $\s$
to the whole complex plane. In order to do so, he proved the following
integral representation of his zeta function
\begin{equation}\label{zeta-integral-original}
\begin{aligned}
K_{-\s}&(x\cdot y)=\frac{1}{\omega_{n-1}}\bigg(\frac{2}{\nu}\bigg)^{\s-1/2}\frac{\Gamma\big(\s+\frac{1}{2}\big)}{\Gamma(2\s)} \\
&\quad\times\int_0^\infty\bigg(\frac{1-e^{-2t}}{(1-2e^{-t}x\cdot y+e^{-2t})^{\nu+1}}-1\bigg)
\mathcal{I}_{\s-1/2}(\nu t)e^{-\nu t}\frac{dt}{t^{1/2-\s}},
\end{aligned}
\end{equation}
where $\nu=\frac{n-2}{2}$ and $\mathcal{I}_{\s-1/2}$ is the modified Bessel function of the first kind
\begin{equation}\label{Bessel-function}
\mathcal{I}_\rho(z)=\sum_{m=0}^\infty\frac{(z/2)^{2m+\rho}}{\Gamma(m+1)\Gamma(m+\rho+1)}.
\end{equation}
Recall that $\mathcal{I}_\rho(z)$ is an analytic function of $z\in\C\setminus(-\infty,0)$
and an entire function of $\rho$. Moreover, the series in \eqref{Bessel-function} is uniformly convergent in any disk $\abs{z}<R$, $\abs{\rho}<N$. See \cite[Chapter~5]{Lebedev} and  \cite{online} for details.
It is easy to see that the integral \eqref{zeta-integral-original} defines an analytic function of $\s$
in the half plane $\hbox{Re}(\s)>0$. Then, Minakshisundaram obtained the meromorphic
continuation of his zeta function by performing the usual trick of replacing
the domain of integration in \eqref{zeta-integral-original} by a loop integral
around the real axis, with a small circle around the origin.

Thanks to Minakshisundaram's integral formula \eqref{zeta-integral-original} we are able
to estimate the integral kernel in the pointwise formula \eqref{integral-formula}.
For the sake of simplicity and because of our interest on fractional nonlocal equations,
we will only consider the case when $\s$ is real, namely,
$\s=s>0$. Recall the semigroup kernel $P_{e^{-t}}$ in \eqref{Poisson semigroup}.

\begin{theorem}[Fractional integrals on the sphere]\label{thm:negative powers any n}
Let $s>0$. Then
$$(-\Delta_{\S^{n-1}})^{-s}u(x)=\int_{\S^{n-1}}K_{-s}(x\cdot y)u(y)\,d\mathcal{H}^{n-1}(y),$$
for $x\in\S^{n-1}$. The kernel $K_{-s}(x\cdot y)$ is given by
\begin{equation}\label{zeta-integral}
\begin{aligned}
K_{-s}(x\cdot y)&=\bigg(\frac{4}{n-2}\bigg)^{s-1/2}\frac{\Gamma\big(s+\frac{1}{2}\big)}{\Gamma(2s)} \\
&\quad\times\int_0^\infty e^{-t(\frac{n-2}{2})}
\big(P_{e^{-t}}(x\cdot y)-\tfrac{1}{\omega_{n-1}}\big)\mathcal{I}_{s-1/2}(\tfrac{n-2}{2}t)\,\frac{dt}{t^{1/2-s}},
\end{aligned}
\end{equation}
and satisfies the estimates
$$|K_{-s}(x\cdot y)|\leq c_{n,s}
\begin{cases}
    1/d(x,y)^{(n-1)-2s}, & \hbox{if}~s<\frac{n-1}{2}, \\
    \ln(1+(1-x\cdot y)^{-n/2}), & \hbox{if}~s=\frac{n-1}{2}, \\
    1, & \hbox{if}~s>\frac{n-1}{2},
\end{cases}$$
for all $x,y\in\S^{n-1}$, $x\neq y$.
\end{theorem}

\begin{proof}
We split the integral in \eqref{zeta-integral} as the sum of two integrals $I+II$, where
$$I=\int_0^1e^{-t(\frac{n-2}{2})}\big(P_{e^{-t}}(x\cdot y)-\tfrac{1}{\omega_{n-1}}\big)\mathcal{I}_{s-1/2}
(\tfrac{n-2}{2}t)\,\frac{dt}{t^{1/2-s}}.$$

Let us estimate $I$. By using the asymptotic behavior
$\mathcal{I}_\rho(z)\approx\frac{1}{\Gamma(\rho+1)}(\frac{1}{2}z)^\rho$, valid for $z\to0$ and
any $\rho\neq-1,-2,\ldots$ (see \cite{Lebedev,online}) we see that
$$\int_0^1 e^{-t(\frac{n-2}{2})}\tfrac{1}{\omega_{n-1}}\mathcal{I}_{s-1/2}(\tfrac{n-2}{2}t)\frac{dt}{t^{1/2-s}}
\sim\int_0^1t^{2s-1}\,dt\sim 1.$$
On the other hand,
\begin{align*}
    \int_0^1 &e^{-t(\frac{n-2}{2})}P_{e^{-t}}(x\cdot y)\mathcal{I}_{s-1/2}(\tfrac{n-2}{2}t)\,\frac{dt}{t^{1/2-s}} \\
     &\sim \int_0^1\frac{1-e^{-2t}}{((1-e^{-t})^2+2e^{-t}(1-x\cdot y))^{n/2}}\,t^{2s-1}\,dt \\
     &\sim \int_0^1\frac{t^{2s}}{(t^2+(1-x\cdot y))^{n/2}}\,dt=:I_1.
\end{align*}
If $s<\frac{n-1}{2}$ then
\begin{align*}
     I_1 &= \frac{1}{(1-x\cdot y)^{n/2}}\int_0^1\frac{t^{2s}}{(\frac{t^2}{1-t}+1)^{n/2}}\,dt \\
     &= \frac{1}{(1-x\cdot y)^{\frac{(n-1)-2s}{2}}}\int_0^{\frac{1}{(1-x\cdot y)^{1/2}}}
     \frac{\omega^{2s}}{(\omega^2+1)^{n/2}}\,d\omega \\
     &\sim \frac{1}{d(x,y)^{(n-1)-2s}}.
\end{align*}
In the case $s=\frac{n-1}{2}$ we notice that
\begin{align*}
    I_1 &= \int_0^1\frac{t^{n-1}}{(t^2+(1-x\cdot y))^{n/2}}\,dt \\
     &\sim \int_0^1\frac{t^{n-1}}{t^n+(1-x\cdot y)^{n/2}}\,dt \\
     &= \frac{1}{n}\ln(1+(1-x\cdot y)^{-n/2}).
\end{align*}
Finally, if $s>\frac{n-1}{2}$ then $2s=n-1+\varepsilon$ for some $\varepsilon>0$ and
$$I_1=\int_0^1\frac{t^{n-1+\varepsilon}}{(t^2+(1-x\cdot y))^{n/2}}\,dt
\sim \int_0^1\frac{t^{n-1+\varepsilon}}{t^n+(1-x\cdot y)^{n/2}}\,dt\sim 1,$$
where the last estimate can be checked by considering the cases $1-x\cdot y<1$ and $1\leq1-x\cdot y<2$.
We conclude that $I$ satisfies estimates as in the statement.

Consider now $II$. For $1<t<\infty$ we have $0<r=e^{-t}<e^{-1}<1$.
We then need to estimate $|P_r(x\cdot y)-\tfrac{1}{\omega_{n-1}}|$ uniformly in $x\cdot y\in[-1,1)$,
for every $0<r<e^{-1}$. Fix any such $r$.
For each $x\cdot y\in[-1,1)$, by the mean value theorem,
$\big|P_{r}(x\cdot y)-\tfrac{1}{\omega_{n-1}}\big|=\abs{P_{r}(x\cdot y)-P_0(x\cdot y)}
=\abs{\partial_rP_r(x\cdot y)}\big|_{r=\xi}r$, for some $\xi$ between $0$ and $r$.
We have
$$\partial_rP_r(x\cdot y)\big|_{r=\xi}
= \frac{-2r+(x\cdot y-r)[4r^2+n(1-r^2)]+2r^3}{\omega_{n-1}((1-r)^2+2r(1-x\cdot y))^{n/2+1}}\Big|_{r=\xi}.$$
It is clear that $|\partial_rP_r(x\cdot y)|\big|_{r=\xi}\leq c_n$,
uniformly in $x\cdot y$. Hence, by going back to $r=e^{-t}$,
we get $|P_{e^{-t}}(x\cdot y)-\frac{1}{\omega_{n-1}}|\leq c_n e^{-t}$, uniformly in $x\cdot y$.
This and the asymptotic expansion
$\mathcal{I}_\rho(z)\approx\frac{1}{(2\pi z)^{1/2}}e^z(1+O(z^{-1}))$ as $|z|\to\infty$
(see \cite{online, Lebedev}) imply
$$II\leq c_{n,s}\int_1^\infty e^{-t(\frac{n-2}{2})}e^{-t}\frac{e^{t(\frac{n-2}{2})}}{t^{1/2}}\,\frac{dt}{t^{1/2-s}}=c_{n,s}.$$
By pasting together the estimates for $I$ and $II$ the conclusions follow.
\end{proof}

\section{Positive powers and the Dirichlet-to-Neumann map}\label{Section:Positive}

In this section we let $\s=s$ to be real, with $0<s<1$. Recall that
\begin{equation}\label{Lap frac serie}
\begin{aligned}
(-\Delta_{\S^{n-1}})^s u(x) &=\sum_{k=0}^\infty\lambda_k^sP_ku(x) \\
&=\sum_{k=0}^\infty(k(k+n-2))^s\sum_{l=1}^{d_k}c_{k,l}(u)Y_{k,l}(x).
\end{aligned}
\end{equation}
If $u\in C^\infty(\S^{n-1})$ then  \eqref{acotacion esf arm} and
\eqref{eq:decaycoefficients} imply that the sums in \eqref{Lap frac serie} are absolutely convergent
and can be differentiated term by term, so $(-\Delta_{\S^{n-1}})^su\in C^\infty(\S^{n-1})$.

Our aim is to prove a numerical identity (Lemma \ref{Lem:Minak positive})
that will permit us to express the operator $(-\Delta_{\S^{n-1}})^s$
in terms of its principal part, namely, the fractional power of the Dirichlet-to-Neumann
map $(L+\frac{n-2}{2})^{2s}$, plus a remainder, smoothing operator $S^s$. Since $L$
is an operator of order one, $(L+\frac{n-2}{2})^{2s}$ is an operator of order $2s$,
which is also the order of $(-\Delta_{\S^{n-1}})^s$. The remainder operator $S^s$
will act as a fractional integral operator of order $2-2s$. The semigroup $e^{-tL}$ generated
by $L$ (see section \ref{Section:Tools}) will play a key role throughout the analysis.

Let us begin with our new numerical identity. Here we were
inspired by the work of Minakshisundaram. Indeed, he proved
a similar formula but for negative values of $s$
in \cite[Lemma~2]{Minakshisundaram},
which allowed him in turn to find \eqref{zeta-integral-original}.

\begin{lemma}\label{Lem:Minak positive}
Let $k\geq0$, $\nu\geq0$, with $k+\nu>0$, and $0<s<1$. Then
$$(k(k+2\nu))^s=(k+\nu)^{2s}+\frac{(2\nu)^{s+1/2}}{\pi^{-1/2}\Gamma(-s)}
\int_0^\infty e^{-t(k+\nu)}\mathcal{B}_{-s-1/2}(\nu t)\,\frac{dt}{t^{1/2+s}}.$$
Here $\mathcal{B}_\rho(z)$ is the modified Bessel function of the first kind $\mathcal{I}_\rho(z)$
minus the first term of its Taylor series expansion, namely,
$$\mathcal{B}_\rho(z)=\mathcal{I}_\rho(z)-\frac{(z/2)^\rho}{\Gamma(\rho+1)}
=\sum_{m=1}^\infty\frac{(z/2)^{2m+\rho}}{\Gamma(m+1)\Gamma(m+\rho+1)}.$$
\end{lemma}

\begin{proof}
The following estimates, direct consequences of the series above and the
asymptotic formula for $\mathcal{I}_\rho(z)$ for large $|z|$, hold:
\begin{equation}\label{asintoticas B}
|\mathcal{B}_{-s-1/2}(\nu t)|\leq c_s
\begin{cases}
\displaystyle(\nu t)^{2-s-1/2},&\hbox{as}~t\to0,\\
\displaystyle\frac{e^{\nu t}}{(\nu t)^{1/2}},&\hbox{as}~t\to\infty.
\end{cases}
\end{equation}
These show that the integral in the statement is absolutely convergent.
Let $|\tau|<1$.
From the binomial theorem and the symmetry formula for the quotient
of Gamma functions we get
\begin{align*}
(1-\tau)^s &=\sum_{m=0}^\infty (-1)^m\frac{\Gamma(s+1)}{\Gamma(m+1)\Gamma(s-m+1)}\,\tau^m \\
&=1+\sum_{m=1}^\infty\frac{\Gamma(m-s)}{\Gamma(m+1)\Gamma(-s)}\,\tau^m
\end{align*}
For any $\lambda,\alpha>0$ it is easy to check that the following identity holds:
$$\lambda^{-\alpha}=\frac{1}{\Gamma(\alpha)}\int_0^\infty e^{-t\lambda}\,\frac{dt}{t^{1-\alpha}}.$$
Let us take $\tau=\frac{\nu^2}{(k+\nu)^2}$, $\lambda=k+\nu$, $\alpha=2(m-s)$,
and apply the duplication formula for the Gamma function (the second identity in \eqref{Gamma})
with $z=m-s$ to obtain
\begin{align*}
    (k(k+2\nu))^s &= (k+\nu)^{2s}\left(1-\frac{\nu^2}{(k+\nu)^2}\right)^s \\
     &= (k+\nu)^{2s}+\sum_{m=1}^\infty\frac{\Gamma(m-s)}{\Gamma(m+1)\Gamma(-s)}\nu^{2m}(k+\nu)^{-2(m-s)} \\
     &= (k+\nu)^{2s} \\
     &\quad+\sum_{m=1}^\infty\frac{\Gamma(m-s)\nu^{2m}}{\Gamma(m+1)\Gamma(-s)\Gamma(2(m-s))}
     \int_0^\infty e^{-t(k+\nu)}\frac{dt}{t^{1-2(m-s)}} \\
     &= (k+\nu)^{2s} \\
     &\quad+\frac{(2\nu)^{s+1/2}}{\pi^{-1/2}\Gamma(-s)}\int_0^\infty e^{-t(k+\nu)}
     \bigg[\sum_{m=1}^\infty\frac{(\nu t/2)^{2m+(-s-1/2)}}{\Gamma(m+1)\Gamma(m-s+1/2)}\bigg]\frac{dt}{t^{1/2+s}}.
\end{align*}
The sum inside the brackets is exactly $\mathcal{B}_{-s-1/2}(\nu t)$.
\end{proof}

Let $u$ be as in \eqref{decomp f}, see also \eqref{projector}.
The fractional power of order $2s$ of the Dirichlet-to-Neumann operator
$L+\tfrac{n-2}{2}$ is given by
\begin{equation}\label{eq:fractional L}
\begin{aligned}
(L+\tfrac{n-2}{2})^{2s}u(x) &= \sum_{k=0}^\infty(k+\tfrac{n-2}{2})^{2s}P_ku(x) \\
&=\sum_{k=0}^\infty(k+\tfrac{n-2}{2})^{2s}\sum_{l=1}^{d_k}c_{k,l}(u)Y_{k,l}(x).
\end{aligned}
\end{equation}

By applying the numerical identity from
Lemma \ref{Lem:Minak positive} with $\nu=\tfrac{n-2}{2}$ to
the spectral definition of $(-\Delta_{\S^{n-1}})^s u(x)$ in \eqref{Lap frac serie}
and recalling the definition of $e^{-tL}u(x)$ in \eqref{Poisson sum},
we finally obtain the desired relation between the fractional Laplacian
on the sphere and the fractional Dirichlet-to-Neumann map \eqref{eq:fractional L}.

\begin{theorem}[Fractional Laplacian on the sphere and fractional Dirichlet-to-Neumann map]\label{Thm:Frac Lap semigroup}
Let $0<s<1$. Then
$$(-\Delta_{\S^{n-1}})^su(x)=(L+\tfrac{n-2}{2})^{2s}u(x)+S^su(x),$$
for $x\in\S^{n-1}$. The operator $S^s$ is given by
$$S^su(x)=\frac{(n-2)^{s+1/2}}{\pi^{-1/2}\Gamma(-s)}\int_0^\infty
e^{-t(\frac{n-2}{2})}e^{-tL}u(x)\mathcal{B}_{-s-1/2}(\tfrac{n-2}{2}t)\,\frac{dt}{t^{1/2+s}}.$$
\end{theorem}

We pointed out in section \ref{Section:Tools} that $-\Delta_{\S^{n-1}}=L$ only when $n=2$.
Notice that if we let $n=2$ in Theorem \ref{Thm:Frac Lap semigroup} then
$S^su=0$ and $(-\Delta_{\T})^su=L^{2s}u$. The fractional Laplacian on the torus,
which includes this case $n=2$, has been extensively studied in \cite{Roncal-Stinga1,Roncal-Stinga2}.
Hence for the rest of this section we will focus on the case $n\geq3$.
We see from Theorem \ref{Thm:Frac Lap semigroup} that in order to 
understand the fractional Laplacian on the sphere in terms of the
fractional powers of the Dirichlet-to-Neumann map, 
we need to study separately $(L+\tfrac{n-2}{2})^{2s}u$ and the fractional integral operator $S^su$.

\subsection{The fractional Dirichlet-to-Neumann operator $(L+\frac{n-2}{2})^{2s}$}

As it was expected, the regularization effects of $(-\Delta_{\S^{n-1}})^s$ are given by
the fractional operator $(L+\tfrac{n-2}{2})^{2s}$. Since $L$ is an operator of order one,
it is enough to restrict our analysis to the case of powers $2s\in(0,1)$,
in which $(L+\frac{n-2}{2})^{2s}$ becomes an operator of order $2s$.
When $2s\in[1,2)$ we can just simply write
$$(L+\tfrac{n-2}{2})^{2s}u=(L+\tfrac{n-2}{2})^{2s-1}(Lu+\tfrac{n-2}{2}u),$$
and $2s-1\in[0,1)$.

\begin{theorem}[Fractional Dirichlet-to-Neumann map]\label{thm:D-t-N}
Let $0<2s<1$. Then
$$(L+\tfrac{n-2}{2})^{2s}u(x)=\int_{\S^{n-1}}(u(x)-u(y))L_{2s}(x\cdot y)\,d\mathcal{H}^{n-1}(y)
+\big(\tfrac{n-2}{2}\big)^{2s}u(x),$$
for $x\in\S^{n-1}$, where the integral is absolutely convergent.
The kernel $L_{2s}(x\cdot y)$ is given by
$$L_{2s}(x\cdot y)=\frac{1}{|\Gamma(-2s)|}\int_0^\infty
e^{-t(\frac{n-2}{2})}P_{e^{-t}}(x\cdot y)\,\frac{dt}{t^{1+2s}},$$
and satisfies the estimate
$$L_{2s}(x\cdot y)\sim\frac{1}{d(x,y)^{(n-1)+2s}},$$
for all $x,y\in\S^{n-1}$, $x\neq y$.
\end{theorem}

\begin{proof}
By applying the numerical formula
\begin{equation}\label{eq:numerical identity}
\lambda^{2s}=\frac{1}{\Gamma(-2s)}\int_0^\infty\big(e^{-t\lambda}-1\big)\,\frac{dt}{t^{1+2s}},
\end{equation}
with $\lambda=(k+\frac{n-2}{2})>0$ to \eqref{eq:fractional L} and recalling \eqref{Poisson sum}
we obtain the semigroup formula
$$(L+\tfrac{n-2}{2})^{2s}u(x)=\frac{1}{\Gamma(-2s)}\int_0^\infty\big(e^{-t(\frac{n-2}{2})}e^{-tL}u(x)
-u(x)\big)\,\frac{dt}{t^{1+2s}}.$$
Lemma \ref{Prop:bound t} shows that this integral is absolutely convergent.
From \eqref{Poisson semigroup} and the fact that
$$e^{-tL}1(x)=\int_{\S^{n-1}}P_{e^{-t}}(x\cdot y)\,d\mathcal{H}^{n-1}(y)\equiv1,$$
for all $t>0$ and $x\in\S^{n-1}$, we arrive to
\begin{equation}\label{prefubini}
\begin{aligned}
(L&+\tfrac{n-2}{2})^{2s}u(x) \\
&=\frac{1}{|\Gamma(-2s)|}\int_0^\infty\int_{\S^{n-1}}(u(x)-u(y))
e^{-t(\frac{n-2}{2})}P_{e^{-t}}(x\cdot y)\,d\mathcal{H}^{n-1}(y)\,\frac{dt}{t^{1+2s}} \\
&\quad +u(x)\frac{1}{\Gamma(-2s)}\int_0^\infty\big(e^{-t(\frac{n-2}{2})}-1\big)\,\frac{dt}{t^{1+2s}}.
\end{aligned}
\end{equation}
The identity in \eqref{eq:numerical identity} with $\lambda=\frac{n-2}{2}$ implies that
the second term in \eqref{prefubini} is equal to $(\frac{n-2}{2})^{2s}u(x)$.
We want to interchange the order of the integrals in the first term of
\eqref{prefubini}. On one hand, since $2s<1$,
\begin{align*}
    \int_0^1e^{-t(\frac{n-2}{2})}P_{e^{-t}}(x\cdot y)\,\frac{dt}{t^{1+2s}}
    &\sim \int_0^1\frac{t}{(t^2+(1-x\cdot y))^{n/2}}\,\frac{dt}{t^{1+2s}} \\
    &\sim\frac{1}{(1-x\cdot y)^{\frac{(n-1)+2s}{s}}}\sim \frac{1}{d(x,y)^{(n-1)+2s}}.
\end{align*}
Then, by the regularity of $u$ and the Funk--Hecke formula \eqref{Funk--Hecke},
\begin{align*}
    \int_0^1\int_{\S^{n-1}}|u(x)-u(y)|&e^{-t(\frac{n-2}{2})}P_{e^{-t}}(x\cdot y)\,d\mathcal{H}^{n-1}(y)\,\frac{dt}{t^{1+2s}} \\
    &\leq c_{n,s}\int_{\S^{n-1}}\frac{(1-x\cdot y)^{1/2}}{(1-x\cdot y)^{\frac{(n-1)+2s}{2}}}\,dy \\
    &= c_{n,s}\int_{-1}^1(1-\tau)^{\frac{(1-2s)-(n-1)}{2}}(1-\tau^2)^{\frac{n-3}{2}}\,d\tau<\infty,
\end{align*}
because $2s<1$. On the other hand,
$$\int_1^\infty e^{-t(\frac{n-2}{2})}P_{e^{-t}}(x\cdot y)\,\frac{dt}{t^{1+2s}}
\sim\int_1^\infty e^{-t(\frac{n-2}{2})}\frac{dt}{t^{1+2s}}\sim 1,$$
which gives
$$\int_1^\infty\int_{\S^{n-1}}|u(x)-u(y)|e^{-t(\frac{n-2}{2})}P_{e^{-t}}(x\cdot y)\,d\mathcal{H}^{n-1}(y)\,
\frac{dt}{t^{1+2s}}\leq c_{n,s}\|u\|_{L^\infty(\S^{n-1})}.$$
Thus the double integral in \eqref{prefubini} is absolutely convergent.
The integral representation in the statement then follows from Fubini's theorem.
In addition, the computations above prove the estimate for the kernel $L_{2s}(x\cdot y)$.
\end{proof}

\subsection{The fractional integral operator $S^s$}

The remaining operator $S^su$ in Theorem \ref{Thm:Frac Lap semigroup}
is a fractional integral operator. This is in some sense consistent with the
numerical identities involved in the proof of Lemma \ref{Lem:Minak positive}.

\begin{theorem}[Fractional integral operator $S^s$]
Let $0<s<1$. Then
$$S^s u(x)=\int_{\S^{n-1}}S_s(x\cdot y)u(y)\,d\mathcal{H}^{n-1}(y),$$
for $x\in\S^{n-1}$. The kernel $S_s(x\cdot y)$ is given by
$$S_s(x\cdot y)=\frac{(n-2)^{s+1/2}}{\pi^{-1/2}\Gamma(-s)}\int_0^\infty
e^{-t(\frac{n-2}{2})}P_{e^{-t}}(x\cdot y)\mathcal{B}_{-s-1/2}(\tfrac{n-2}{2}t)\,\frac{dt}{t^{1/2+s}},$$
and satisfies the estimate
$$|S_{s}(x\cdot y)|\leq \frac{c_{n,s}}{d(x,y)^{(n-1)-(2-2s)}},$$
for all $x,y\in\S^{n-1}$, $x\neq y$.
\end{theorem}

\begin{proof}
The spherical convolution formula for $S_su(x)$ in the statement can be derived by using the
kernel representation of $e^{-tL}u(x)$ and Fubini's theorem.
To prove the estimate for the kernel we proceed
as we did for the kernel of the negative powers $(-\Delta_{\S^{n-1}})^{-s}$ in the
proof of Theorem \ref{thm:negative powers any n}.
That is, we split the integral that defines $S_{s}(x\cdot y)$ as the sum of two integrals $I+II$
and we follow analogous computations.
Recall that $n\geq3$. By the estimates for $\mathcal{B}_{s-1/2}$ in \eqref{asintoticas B},
\begin{align*}
    |I| &\leq c_{n,s}\int_0^1\frac{t^{2-2s}}{t^n+d(x,y)^{n}}\,dt \\
     &\leq \frac{c_{n,s}}{d(x,y)^{n}}\int_0^1\frac{t^{2-2s}}{(\frac{t}{d(x,y)})^n+1}\,dt \\
     &=\frac{c_{n,s}}{d(x,y)^{(n-1)-(2-2s)}}\int_0^{\frac{1}{d(x,y)}}\frac{\omega^{2-2s}}{\omega^n+1}\,d\omega \\
     &\leq \frac{c_{n,s}}{d(x,y)^{(n-1)-(2-2s)}},
\end{align*}
because $2-2s-n+1<0$. For the integral $II$, notice that 
$|P_{e^{-t}}(x\cdot y)|\leq c_n$, uniformly in $x\cdot y$, for all $t>1$.
With this and \eqref{asintoticas B} we get
$$|II|\leq c_{n,s}\int_1^\infty e^{-t(\frac{n-2}{2})}\frac{e^{t(\frac{n-2}{2})}}{t^{1/2}}\,\frac{dt}{t^{1/2+s}}=c_{n,s}.$$
\end{proof}

\section{Fractional Laplacians and the heat semigroup on the sphere. \\ Extension problem
and Harnack inequality}\label{Section:extension}

We have shown how the Minakshisudaram zeta function and our
Lemma \ref{Lem:Minak positive}, in combination
with a careful manipulation of the kernel $P_{e^{-t}}(x\cdot y)$ of the semigroup $e^{-tL}u$
generated by the Dirichlet-to-Neumann map $L$,
permit us to obtain integro-differential formulas for fractional powers of the
Laplacian on the sphere, with precise kernel estimates.
In this section we present another technique to finding
pointwise formulas, namely, by means of the method of heat semigroups. This method
was first introduced in \cite{Stinga-Torrea-CPDE} and later on extended
to the most general case in \cite{Gale-Miana-Stinga}.
With this point of view we can also prove an extension problem characterization,
which implies the interior Harnack inequality.
The reader should recall that in section \ref{Section:One-dimensional-positive}
we used the heat semigroup (which in this case is the Jacobi
theta function) to analyze the kernel of the fractional Laplacian on the unit circle
in connection with the Hurwitz zeta function.
Further applications of these ideas
can be found, for example, in \cite{Caffarelli-Silvestre CPDE, Caffarelli-Stinga,
Caffarelli-Vasseur, Roncal-Stinga1, Stinga-Zhang}.

\subsection{Fractional Laplacians on the sphere and the heat semigroup}

The solution $v=v(t,x)$ to the heat equation on the sphere
$$\begin{cases}
\partial_tv=\Delta_{\S^{n-1}}v,&\hbox{for}~t>0,~x\in\S^{n-1},\\
v(0,x)=u(x),&\hbox{for}~x\in\S^{n-1},
\end{cases}$$
is given by the heat semigroup generated by $-\Delta_{\S^{n-1}}$:
\begin{equation}\label{heat semigroup spectral}
\begin{aligned}
e^{t\Delta_{\S^{n-1}}}u(x)\equiv v(t,x) &=\sum_{k=0}^\infty e^{-t\lambda_k}P_ku(x) \\
&=\sum_{k=0}^\infty e^{-t(k(k+n-2))}\sum_{l=0}^{d_k}
c_{k,l}(u)Y_{k,l}(x),
\end{aligned}
\end{equation}
where $P_k$ is the orthogonal projector \eqref{projector}.
By writing down the definition of $c_{k,l}(u)$ and interchanging sums and integral,
we see that the heat semigroup on the sphere
can be written as a spherical convolution
$$e^{t\Delta_{\S^{n-1}}}u(x)=\int_{\S^{n-1}}W_t(x\cdot y)u(y)\,d\mathcal{H}^{n-1}(y).$$
The heat kernel $W_t(x\cdot y)$ is given by
\begin{align*}
W_t(x\cdot y) &= \sum_{k=0}^\infty e^{-t(k(k+n-2))}\sum_{l=0}^{d_k}Y_{k,l}(x)Y_{k,l}(y) \\
&=\frac{1}{\omega_{n-1}}\sum_{k=0}^\infty e^{-t(k(k+n-2))}\frac{k+\frac{n-2}{2}}{\frac{n-2}{2}}\,C_k(x\cdot y),
\end{align*}
where in the second identity we applied \eqref{eq:super identity}.
This kernel turns out to be a smooth, positive function of $t>0$, $x,y\in\S^{n-1}$,
see \cite[Chapter~5]{Davies}. Moreover, for any $t>0$ and $x\in\S^{n-1}$,
\begin{equation}\label{eq:heat kernel at 1}
e^{t\Delta_{\S^{n-1}}}1(x)=\int_{\S^{n-1}}W_t(x\cdot y)\,d\mathcal{H}^{n-1}(y)\equiv1.
\end{equation}
In addition, $W_t(x\cdot y)$ satisfies two-sided Gaussian estimates. In fact,
there is a constant $C>0$ that depends only on $n$ such that
\begin{equation}\label{upper}
W_t(x\cdot y)\leq\frac{C}{t^{(n-1)/2}}e^{-d(x,y)^2/(8t)},
\end{equation}
and 
\begin{equation}\label{lower}
W_t(x\cdot y)\geq\frac{C^{-1}}{t^{(n-1)/2}}e^{-d(x,y)^2/(4t)},
\end{equation}
for all $t>0$, for any $x,y\in\S^{n-1}$, see \cite[Theorems~5.5.6~and~5.6.1]{Davies}.

The heat semigroup formulas for the fractional operators are obtained as follows.
For any $\lambda\geq0$ and $0<s<1$ we have
$$\lambda^s=\frac{1}{\Gamma(-s)}\int_0^\infty\big(e^{-t\lambda}-1\big)\,\frac{dt}{t^{1+s}}.$$
By taking $\lambda=\lambda_k=k(k+n-2)$ and recalling the spectral definition of
the fractional Laplacian on the sphere \eqref{spectral-formula}
we immediately see that
\begin{equation}\label{eq:semigroup formula}
(-\Delta_{\S^{n-1}})^{s}u(x)=\frac{1}{\Gamma(-s)}\int_0^\infty
\big(e^{t\Delta_{\S^{n-1}}}u(x)-u(x)\big)\,\frac{dt}{t^{1+s}}.
\end{equation}
For the negative fractional powers we use that, for any $\lambda,s>0$,
$$\lambda^{-s}=\frac{1}{\Gamma(s)}\int_0^\infty e^{-t\lambda}\,\frac{dt}{t^{1-s}}.$$
If $\displaystyle\int_{\S^{n-1}}u=0$ then we can take $\lambda=\lambda_k$
in the formula above and use the spectral definition \eqref{spectral-formula} to infer that
\begin{equation}\label{eq:semigroup formula negativa}
(-\Delta_{\S^{n-1}})^{-s}u(x)=\frac{1}{\Gamma(s)}\int_0^\infty
e^{t\Delta_{\S^{n-1}}}u(x)\,\frac{dt}{t^{1-s}}.
\end{equation}

\begin{theorem}[Fractional Laplacians and heat semigroup]\label{thm:fractionalandheat}
Let $s>0$.
\begin{enumerate}[$(1)$]
\item If $0<s<1/2$ then
$$(-\Delta_{\S^{n-1}})^{s}u(x)=\int_{\S^{n-1}}(u(x)-u(y))K_{s}(x\cdot y)\,d\mathcal{H}^{n-1}(y),$$
where the integral is absolutely convergent. If $1/2\leq s<1$ then
\begin{align*}
(-\Delta_{\S^{n-1}})^{s}u(x) &= \PV\int_{\S^{n-1}}(u(x)-u(y))K_{s}(x\cdot y)\,d\mathcal{H}^{n-1}(y) \\
&=\int_{\S^{n-1}}(u(x)-u(y)-\nabla_{\S^{n-1}}u(x)\cdot(x-y))K_{s}(x\cdot y)\,d\mathcal{H}^{n-1}(y),
\end{align*}
where the second integral is absolutely convergent.
In both cases the kernel $K_s(x\cdot y)$ is given by
$$K_{s}(x\cdot y)=\frac{1}{|\Gamma(-s)|}\int_0^\infty W_t(x\cdot y)\,\frac{dt}{t^{1+s}}>0,$$
and satisfies the estimate
\begin{equation}\label{eq:kernel estimate}
K_s(x\cdot y)\sim\frac{1}{d(x,y)^{(n-1)+2s}},
\end{equation}
for all $x,y\in\S^{n-1}$, $x\neq y$.
\item If $0<s<\frac{n-1}{2}$ and $\displaystyle\int_{\S^{n-1}}u=0$ then
$$(-\Delta_{\S^{n-1}})^{-s}u(x)=\int_{\S^{n-1}}K_{-s}(x\cdot y)u(y)\,d\mathcal{H}^{n-1}(y),$$
as in Theorem \ref{thm:negative powers any n}.
An equivalent formula for the kernel $K_{-s}(x\cdot y)$ in \eqref{zeta-integral} is
$$K_{-s}(x\cdot y)=\frac{1}{\Gamma(s)}\int_0^\infty W_t(x\cdot y)\,\frac{dt}{t^{1-s}}>0,$$
from which we can deduce the estimate
\begin{equation}\label{eq:kernel estimate negativa}
K_{-s}(x\cdot y)\sim\frac{1}{d(x,y)^{(n-1)-2s}},
\end{equation}
for all $x,y\in\S^{n-1}$, $x\neq y$.
\end{enumerate}
\end{theorem}

\begin{proof}
Let us begin by proving (1).
We first estimate the kernel $K_s(x\cdot y)$ by applying \eqref{upper}--\eqref{lower}.
Indeed, by using the change of variables $r=d(x,y)^2/(8t)$,
\begin{align*}
\int_0^\infty W_t(x\cdot y)\,\frac{dt}{t^{1+s}} &\leq C\int_0^\infty\frac{1}{t^{(n-1)/2}}e^{-d(x,y)^2/(8t)}\,\frac{dt}{t^{1+s}} \\
&= \frac{c_{n,s}}{d(x,y)^{(n-1)+2s}}\int_0^\infty e^{-r}r^{(n-1)/2+s}\,\frac{dr}{r} \\
&=\frac{c_{n,s}}{d(x,y)^{(n-1)+2s}}.
\end{align*} 
The lower bound follows analogously via the change of variables $r=d(x,y)^2/(4t)$.
Therefore \eqref{eq:kernel estimate} follows.

From the semigroup formula \eqref{eq:semigroup formula} and by using \eqref{eq:heat kernel at 1} we get
\begin{equation}\label{eq:bla}
(-\Delta_{\S^{n-1}})^s u(x)=\frac{1}{|\Gamma(-s)|}\int_0^\infty\int_{\S^{n-1}}
(u(x)-u(y))W_t(x\cdot y)d\mathcal{H}^{n-1}(y)\frac{dt}{t^{1+s}}.
\end{equation}
The pointwise formulas in $(1)$ will follow after we apply Fubini's theorem.

Suppose that $0<s<1/2$. Then the double integral in \eqref{eq:bla}
is absolutely convergent because of \eqref{eq:kernel estimate} and
the estimate $|u(x)-u(y)|\leq C|x-y|\sim Cd(x,y)$. Thus
Fubini's theorem can be applied to conclude.

Suppose next that $1/2\leq s<1$. Observe that 
$W_t(x\cdot y)=F(|x-y|)$ for some function $F:\R\to\R$. This implies that
\begin{equation}\label{cero}
\int_{\S^{n-1}\setminus B_\varepsilon(x)}(x_i-y_i)W_t(x\cdot y)\,d\mathcal{H}^{n-1}(y)=0,
\end{equation}
for any $i=1,\ldots,n$, for every $\varepsilon\geq 0$.
Hence the double integral in \eqref{eq:bla} can be written as
$$\int_0^\infty\int_{\S^{n-1}}
(u(x)-u(y)-\nabla u(x)(x-y))W_t(x\cdot y)d\mathcal{H}^{n-1}(y)\frac{dt}{t^{1+\s}}.$$
By \eqref{eq:kernel estimate} and
the estimate $|u(x)-u(y)-\nabla u(x)(x-y)|\leq C|x-y|^2\sim Cd(x,y)^2$, this double
integral is absolutely convergent, so that Fubini's theorem can be used. The principal value
formula follows by applying \eqref{cero} for each $\varepsilon>0$.

Let us continue with the proof of $(2)$.
The pointwise formula for $(-\Delta_{\S^{n-1}})^{-s}u(x)$
follows by writing down the heat kernel in \eqref{eq:semigroup formula negativa}
and using Fubini's theorem. This shows also the equivalent heat semigroup formula for
the kernel $K_{-s}(x\cdot y)$.
The estimate in \eqref{eq:kernel estimate negativa} is obtained
by using the heat kernel bounds as above. We just do the upper bound, the lower bound
is completely analogous. We have
\begin{align*}
K_{-s}(x\cdot y) &\leq C\int_0^\infty\frac{1}{t^{(n-1)/2}}e^{-d(x,y)^2/(8t)}\,\frac{dt}{t^{1-s}} \\
&= \frac{c_{n,s}}{d(x,y)^{(n-1)-2s}}\int_0^\infty e^{-r}r^{(n-1)/2-s}\,\frac{dr}{r},
\end{align*}
and the last integral is finite because $(n-1)/2-s>0$.
\end{proof}

\subsection{Extension problem and Harnack inequality}

The extension problem for the fractional Laplacian on the sphere is a particular
case of the general extension problem proved in \cite{Stinga-Torrea-CPDE}, see also \cite{Gale-Miana-Stinga}.
We present the proof here for the convenience of the reader. As an application,
the interior Harnack inequality is proved.

\begin{theorem}[Extension problem]\label{thm:extension}
Let $0<s<1$. Define
$$U(x,y)=\frac{y^{2s}}{4^s\Gamma(s)}\int_0^\infty e^{-y^2/(4t)}e^{t\Delta_{\S^{n-1}}}u(x)\,\frac{dt}{t^{1+s}},$$
for $x\in\S^{n-1}$, $y>0$. Then $U$ solves
$$\begin{cases}
\Delta_{\S^{n-1}}U+\frac{1-2s}{y}\partial_yU+\partial_{yy}U=0,&\hbox{for}~x\in\S^{n-1},~y>0,\\
U(x,0)=u(x),&\hbox{for}~x\in\S^{n-1},\\
-y^{1-2s}\partial_yU(x,y)\big|_{y=0^+}=\frac{\Gamma(1-s)}{4^{s-1/2}\Gamma(s)}
(-\Delta_{\S^{n-1}})^s u(x),&\hbox{for}~x\in\S^{n-1}.
\end{cases}$$
Moreover,
if $\displaystyle\int_{\S^{n-1}}u=0$ then
\begin{equation}\label{eq:g}
U(x,y)=\frac{1}{\Gamma(s)}\int_0^\infty e^{-y^2/(4t)}e^{t\Delta_{\S^{n-1}}}
\big((-\Delta_{\S^{n-1}})^s u\big)(x)\,\frac{dt}{t^{1-s}}.
\end{equation}
\end{theorem}

\begin{remark}[Extension problem for negative powers]
Let $f$ be a function on the sphere such that $\displaystyle\int_{\S^{n-1}}f=0$.
Consider the solution $u$ to
$$(-\Delta_{\S^{n-1}})^su=f,\quad0<s<1,$$
such that $\displaystyle\int_{\S^{n-1}}u=0$. Then \eqref{eq:g} reads
$$U(x,y)=\frac{1}{\Gamma(s)}\int_0^\infty e^{-y^2/(4t)}e^{t\Delta_{\S^{n-1}}}
f(x)\,\frac{dt}{t^{1-s}},$$
which solves the Neumann problem
$$\begin{cases}
\Delta_{\S^{n-1}}U+\frac{1-2s}{y}\partial_yU+\partial_{yy}U=0,&\hbox{for}~x\in\S^{n-1},~y>0,\\
-y^{1-2s}\partial_yU(x,y)\big|_{y=0^+}=\frac{\Gamma(1-s)}{4^{s-1/2}\Gamma(s)}f(x),&\hbox{for}~x\in\S^{n-1}.
\end{cases}$$
This is the extension problem for $(-\Delta_{\S^{n-1}})^{-s}$. Indeed,
$$U(x,0)=u(x)=(-\Delta_{\S^{n-1}})^{-s}f(x).$$
\end{remark}

\begin{proof}[Proof of Theorem \ref{thm:extension}]
By using the change of variables $y^2/(4t)=r$ in the definition of $U$, we have the equivalent formula
$$U(x,y)=\frac{1}{\Gamma(s)}\int_0^\infty e^{-r}e^{\frac{y^2}{4r}\Delta_{\S^{n-1}}}u(x)\,\frac{dr}{r^{1-s}},$$
from which immediately follows that $U(x,0)=u(x)$. Let $\mathcal{K}_\rho(z)$ denote the modified Bessel function
of second kind of order $\rho$, see \cite{Lebedev,online}.
By using its integral representation in \cite[eq.~(5.10.25)]{Lebedev} 
and the spectral definition of $e^{t\Delta_{\S^{n-1}}}u(x)$ in \eqref{heat semigroup spectral} we can write
\begin{align*}
U(x,y) &= \frac{y^{2s}}{4^s\Gamma(s)}\sum_{k=0}^\infty\bigg[\int_0^\infty e^{-y^2/(4t)}e^{-t\lambda_k}
\,\frac{dt}{t^{1+s}}\bigg]P_ku(x) \\
&= \frac{2^{1-s}}{\Gamma(s)}\sum_{k=0}^\infty (y\lambda_k^{1/2})^s\mathcal{K}_s(y\lambda_k^{1/2})P_ku(x) \\
&\equiv \frac{2^{1-s}}{\Gamma(s)}(y(-\Delta_{\S^{n-1}})^{1/2})^{s}\mathcal{K}_s(y(-\Delta_{\S^{n-1}})^{1/2})u(x).
\end{align*}
With any of these formulas and the identities for the derivatives of $\mathcal{K}_s(z)$
it is easy to check that $U$ satisfies the extension equation.
By noticing that 
$$\frac{y^{2s}}{4^s\Gamma(s)}\int_0^\infty e^{-y^2/(4t)}\frac{dt}{t^{1+s}}=1,$$
we get
\begin{align*}
-y^{1-2s}\partial_yU(x,y) &= \frac{1}{4^s\Gamma(s)}\int_0^\infty
\bigg(\frac{y^2}{2t}-2s\bigg)e^{-y^2/(4t)}e^{t\Delta_{\S^{n-1}}}u(x)\,\frac{dt}{t^{1+s}} \\
&=\frac{1}{4^s\Gamma(s)}\int_0^\infty
\bigg(\frac{y^2}{2t}-2s\bigg)e^{-y^2/(4t)}\big(e^{t\Delta_{\S^{n-1}}}u(x)-u(x)\big)\frac{dt}{t^{1+s}} \\
&\longrightarrow -\frac{2s}{4^s\Gamma(s)}\int_0^\infty
\big(e^{t\Delta_{\S^{n-1}}}u(x)-u(x)\big)\,\frac{dt}{t^{1+s}} \\
&\qquad =\frac{\Gamma(1-s)}{4^{s-1/2}\Gamma(s)}(-\Delta_{\S^{n-1}})^s u(x),\quad\hbox{as}~y\to0^+,
\end{align*}
where in the last identity we used the semigroup formula \eqref{eq:semigroup formula}.
If $\displaystyle\int_{\S^{n-1}}u=0$ then $P_0u(x)=0$ and, by the change of variables $r=y^2/(4t\lambda_k)$,
$k\geq1$, we obtain
\begin{align*}
U(x,y) &= \frac{y^{2s}}{4^s\Gamma(s)}\sum_{k=1}^\infty\bigg[\int_0^\infty e^{-y^2/(4t)}e^{-t\lambda_k}
\,\frac{dt}{t^{1+s}}\bigg]P_ku(x) \\
&=\frac{1}{\Gamma(s)}\sum_{k=1}^\infty\bigg[\int_0^\infty e^{-y^2/(4r)}e^{-r\lambda_k}\lambda_k^sP_ku(x)
\,\frac{dt}{t^{1+s}}\bigg],
\end{align*}
which gives \eqref{eq:g}.
\end{proof}

\begin{theorem}[Harnack inequality]
Let $\Omega'\subset\subset\Omega\subset\S^{n-1}$ be open sets. There is
a constant $C>0$ depending on $\Omega'$, $\Omega$, $n$ and $s$ such that
$$\sup_{\Omega'}u\leq C\inf_{\Omega'}u,$$
for any solution $u$ to
$$\begin{cases}
(-\Delta_{\S^{n-1}})^su=0,&\hbox{in}~\Omega,\\
u\geq0,&\hbox{in}~\S^{n-1},
\end{cases}$$
\end{theorem}

\begin{proof}
The idea is to use the extension problem as done in other contexts, see for example
\cite{Caffarelli-Silvestre CPDE, Roncal-Stinga1, Stinga-Torrea-CPDE, Stinga-Zhang}.
We sketch the steps next. Details are left to the interested reader.
For any $u$ as in the statement, let $U$ be the solution to the extension problem
given by Theorem \ref{thm:extension}. Since $u\geq0$ in $\S^{n-1}$ and the heat kernel $W_t(x\cdot y)$
is positive, we have $e^{t\Delta_{\S^{n-1}}}u\geq0$ and thus
$U\geq0$ in $\S^{n-1}\times[0,\infty)$. The extension equation can be
written as
$$\operatorname{div}_{(\S^{n-1},y)}(y^{1-2s}\nabla_{(\S^{n-1},y)}U)=0,$$
where $\nabla_{(\S^{n-1},y)}=(\nabla_{\S^{n-1}},\partial_y)$ and
$\operatorname{div}_{(\S^{n-1},y)}$ is the corresponding divergence operator.
Let $\bar{U}(x,y)=U(x,|y|)\geq0$, for $x\in\S^{n-1}$ and $y\in\R$. By using that
$$-y^{1-2s}\partial_yU(x,y)\big|_{y=0^+}=\frac{\Gamma(1-s)}{4^{s-1/2}\Gamma(s)}(-\Delta_{\S^{n-1}})^su=0,
\quad\hbox{in}~\Omega,$$
it can be checked that $\bar{U}$ is a weak solution the the degenerate elliptic equation
$$\operatorname{div}_{(\S^{n-1},y)}(|y|^{1-2s}\nabla_{(\S^{n-1},y)}\bar{U})=0,$$
with Muckenhoupt weight $\omega(x,y)=|y|^{1-2s}\in A_2$, in $\{(x,y):x\in\Omega,-1/2<y<1/2\}$.
These degenerate elliptic equations admit interior Harnack inequality (see \cite{Fabes}),
so there exists a constant $C>0$ depending only on $\Omega$, $\Omega'$, $n$ and $s$ such that
$$\sup_{\Omega'\times(-1/4,1/4)}\bar{U}\leq C\inf_{\Omega'\times(-1/4,1/4)}\bar{U}.$$
By restricting $\bar{U}$ back to $y=0$ the Harnack inequality for $u$ follows.
\end{proof}

\noindent\textbf{Acknowledgements.} The authors would like to thank Sundaram Thangavelu
for providing a copy of Minakshisundaram's paper \cite{Minakshisundaram}. They are
also grateful to the referee for useful comments and questions.

\bibliographystyle{amsplain}



\end{document}